\documentclass[12pt, reqno]{amsart}
\usepackage{amsmath, amsthm, amscd, amsfonts, amssymb, graphicx, color, mathrsfs}
\usepackage[bookmarksnumbered, colorlinks, plainpages]{hyperref}
\hypersetup{colorlinks=true,linkcolor=red, anchorcolor=green, citecolor=cyan, urlcolor=red, filecolor=magenta, pdftoolbar=true}

\let\<=\langle
\let\>=\rangle

\def\smallfrac#1#2{{\textstyle{\frac{#1}{#2}}}}

\def\M{{\mathcal M}}
\def\H{{\mathcal H}}

\def\BB{{\mathbb B\kern.5pt}}
\def\CC{{\mathbb C\kern.5pt}}
\def\NN{{\mathbb N\kern.5pt}}
\def\MM{{\mathbb M\kern.5pt}}
\def\RR{{\mathbb R\kern.5pt}}
\def\ZZ{{\mathbb Z\kern.5pt}}

\newtheorem{theorem}{Theorem}[section]
\newtheorem{lemma}[theorem]{Lemma}
\newtheorem{proposition}[theorem]{Proposition}
\newtheorem{corollary}[theorem]{Corollary}
\theoremstyle{definition}
\newtheorem{definition}[theorem]{Definition}

\theoremstyle{remark}
\newtheorem{remark}[theorem]{Remark}
\numberwithin{equation}{section}

\begin{document}
\title[On the Commutant of Invertible Positive Operators]{On the Commutant of Invertible Positive Operators}
\author[F. Kittaneh]{Fuad Kittaneh$^{1.2}$}
\address{$^{1}$Department of Mathematics, The University of Jordan, Amman, Jordan\newline
\indent $^2$ Department of Mathematics, Korea University, Seoul 02841, South Korea}
\email{fkitt@ju.edu.jo}

\author[C.S. Kubrusly]{Carlos S. Kubrusly$^3$}
\address{$^3$Catholic University of Rio de Janeiro, Rio de Janeiro, Brazil}
\email{carlos@ele.puc-rio.br}

\author[M.S. Moslehian]{Mohammad Sal Moslehian$^4$}
\address{$^4$Department of Pure Mathematics, Faculty of Mathematical Sciences, Ferdowsi University of Mashhad, P. O. Box 1159, Mashhad 91775, Iran}
\email{moslehian@um.ac.ir; moslehian@yahoo.com; msmoslehian@gmail.com}


\begin{abstract}
An important result due to Radjabalipour asserts that if ${A\in\BB(\H)}$ is an invertible positive operator and ${T\in\BB(\H)}$ is arbitrary, then the operator sequence $\{A^{-n}TA^n\}_{n\in\ZZ}$ is norm-bounded if and only if ${TA=AT}.$ Using the concept of spectral gaps, we show that if $A$ possesses a spectral gap and the sequence
$\{A^{-n}TA^n\}_{n\in\ZZ}$ is norm-bounded, then $T$ must be block diagonal with respect to a certain orthogonal decomposition of $\H$. We also provide an alternative proof of Radjabalipour's theorem. We also extend the theorem to some rich classes of operators, including invertible normal operators and invertible weighted sums of projections. Moreover, we provide simple proofs for the fact that a positive definite matrix ${A\in\MM_N}$ is a positive multiple of the identity if and only if the sequence $\{A^{-n}TA^n\}_{n\in\ZZ}$ is norm-bounded for every~${T\in\MM_N}.$
\end{abstract}

\maketitle

\section{Introduction}
Radjabalipour \cite[Corollary 3]{RAD} has investigated Foia\c{s}' proof
\cite{FOI} of Arveson's theorem \cite{ARV} and established the following
theorem.

\begin{theorem}{$\!($}Radjabalipour{$)\!$}\label{thm:1.1}
Let $\kern-1ptA$ be an invertible positive operator on a Hilbert space $\H.$
Then an operator ${T\in\BB(\H)}$ commutes with $A$ if and only if
\[
{\sup}_{n\in\mathbb{Z}}\|A^{-n}TA^n\|<\infty.
\]
\end{theorem}

The original proof in \cite{RAD} is based on an investigation of the behaviour
of the iterates of the net (a bisequence, actually) ${\|A^{-n}TA^n\|}$ as $n$
runs over all integers~$\mathbb{Z}.$

Note that the above result does not hold if the set of all integers $\ZZ$ is
replaced with the set of positive integers $\NN.$ Indeed,
if
${A=\begin{bmatrix} 1 & 0 \\ 0 & 2 \end{bmatrix}}$
and
${T=\begin{bmatrix} 0 & 0 \\ 1 & 0 \end{bmatrix}}\!,$
$ $
then ${\sup_{n\in\NN}\|A^{-n}TA^n\|<\infty}$ and ${AT\ne TA}$.

\vskip6pt
$\!$The study of operators $T$ for which the two-sided \emph{conjugation orbit}
$\{A^{-n}TA^n\!\}_{n\in\NN}$ is bounded was initiated by Deddens \cite{DM}.\
Conjugation orbits have been investigated in several contexts, including
operator algebras, similarity theory, and \hbox{invariant}~sub\-spaces; see,
e.g., \cite{DM,FOI,RAD}.\ They are also closely
related~to~\hbox{dynamical}~proper\-ties of inner automorphisms of $\BB(\H).$
Indeed, if one considers the map ${\Phi_A(T)=A^{-1}TA}$, then
${\Phi_A^{\,n}(T)=A^{-n}TA^n\!}$, so that the sequence $\{A^{-n}TA^n\!\}$ may
be viewed as the~\hbox{orbit} of $T$ under the discrete dynamical system
generated by $\Phi_A$. Moreover, the commutation relation ${A\,T\!=\!TA}$ is
equivalent to the fixed-point equation ${\Phi_A(T)=T}.$~Questions concerning
the boundedness and asymptotic behaviour of such orbits are often subtle and
strongly dependent on the spectral and geometric properties of the
implementing operator $A.$

Observe that ${\|A^{-n}TA^n\|\le\|A^{-n}\|\,\|T\|\,\|A^n\|}$ for every
${n\in\ZZ}.$ Thus, Theorem~\ref{thm:1.1} gives still another proof that the
identity is the unique power-bounded positive operator with a power-bounded
inverse.

\begin{corollary}\label{corcar}
The identity is the unique power-bounded positive operator with a
power-bounded inverse.
\end{corollary}

In fact, let ${T\in\BB(H)}$ be an arbitrary operator.\ Since both
$\{\|A^n\|\}_{n \in\ZZ}$ and $\{\|A^{-n}\|\}_{n \in\ZZ}$ are bounded, the
boundedness of $\{A^{-n}T A^n\}_{n \in\ZZ}$ implies that $T$ commutes with
$A.$ Therefore, $A$ belongs to the commutant of $\BB(\H)$, which is ${\CC\,I}.$
From the power-boundedness of $A$ and $A^{-1}$, we conclude that ${A=I}$.

\begin{remark}
A classical proof of Corollary \ref{corcar} can be obtained by using the
spectral radius:\ Assume $A$ is a positive invertible operator on a Hilbert
space $\H$ such that both $A$ and $A^{-1}$ are power-bounded. Then
\[
r(A) = \lim_{n \to \infty} \|A^n\|^{1/n}
\leq \lim_{n \to \infty} \bigl
(\sup_{n\in \mathbb{Z}}\|A^n\|\bigr)^{1/n} = 1.
\]
Similarly, $r(A^{-1}) \leq 1.$ In addition,
\[
1 = r(I) = r(AA^{-1}) \leq r(A) \, r(A^{-1}) \leq 1,
\]
using the submultiplicativity of the spectral radius, which holds for
commuting operators.\ Since $A$ is positive, its spectral radius coincides
with its operator norm, we have ${A\le\|A\|I=r(A)I=I}$ and 
${A^{-1}\le\|A^{-1}\|I=r(A^{-1})I=I}.$ Thus, ${A=I}$.
\end{remark}

In general, the quantity $\sup_{n\in\ZZ}\|A^{-n}\|\,\|A^n\|$ may be
unbounded.\ For instance, if ${A=\operatorname{diag}(r,r^{-1})}$, where
${r>1}$, then ${\|A^{-n}\|\,\|A^n\|=r^{2n}\to\infty}$ as ${n\to\infty}$.

Moreover, the conjugation orbit itself may fail to be bounded. Consider
$A=\begin{bmatrix} 1 & 1 \\ 0 & 1 \end{bmatrix}$
and
${T=\begin{bmatrix} 0 & 0 \\ 1 & 0 \end{bmatrix}}.$
Since
${A^n=\begin{bmatrix} 1 & n \\ 0 & 1 \end{bmatrix}}$
for all ${n\in\ZZ}$, we get
\[
\sup_{n\in\mathbb{Z}}\|A^{-n}TA^n\|
= \sup_{n\in\mathbb{Z}}\left\|\begin{bmatrix}-n & -n^2 \\
1 & n\end{bmatrix}\right\| = \sup_{n\in\mathbb{Z}} (n^2+1)=\infty.
\]
Thus, boundedness of conjugation orbits is a genuinely nontrivial phenomenon.

Still along the lines of the previous paragraph, it is also worth noticing
that an immediate consequence of the so-called Corach--Porta--Recht inequality
\cite{CPR} (see also \cite{CMS, KIT} and the references therein) asserts that
if $A$ is invertible and self-adjoint,~then
\[
\|A^{-n}TA^n+A^nTA^{-n}\| \geq 2\|T\|, \qquad n\in\mathbb Z.
\]
Nevertheless, no analogous lower bound holds for the individual terms
${\|A^{-n}TA^n\|}.$ Indeed, if
${A=\!\begin{bmatrix} 1 & 0 \\ 0 & 2 \end{bmatrix}\!}$
and
${T\kern-1pt=\!\begin{bmatrix} 0 & 1 \\ 0 & 0 \end{bmatrix}\!},$
then
${A^{-n}TA^n\!=\!\begin{bmatrix} 0 & 2^n \\ 0 & 0 \end{bmatrix}\!}$,
so that ${\|A^{-n}TA^n\|}$ $=2^n$ for all ${n\in\ZZ}.$ Consequently,
\[
\inf_{n\in\mathbb Z}\|A^{-n}TA^n\|
= \inf_{n\in\mathbb Z} 2^n = 0 \not\geq 1 = \|T\|.
\]

Notation and terminology used in this note are standard, and, in general, they
follow those in \cite{CON, FUR, KUB, MUR} for the theory of operators on
Hilbert spaces, and in \cite{BHA, ZHANG} for the theory of matrices. Let
$\BB(\H)$ denote the $C^*$-algebra of all bounded linear operators on a complex
Hilbert space $(\H,\langle\cdot,\cdot\rangle)$ endowed with the operator norm
$\|\cdot\|.$ Capital letters are used to denote operators or matrices.\ In
particular, $I$ and, when convenient, $I_{\H}$ denote the identity operator on
$\H.$ The notation $\oplus$ stands for orthogonal direct sum, either of closed
subspaces or of operators, as usual.\ The spectrum and the spectral radius of
an operator $S$ will be denoted by $\sigma(S)$ and $r(S)$, respectively.\
Recall that a closed subspace $\M$ of a Hilbert space $\H$ reduces an operator
${S\in\BB(\H)}$ if ${S(\M)\subseteq\M}$ and ${S^*(M)\subseteq\M}.$ In this
case, the restriction $S|_\M$ of $S$ to $\M$ can be regarded as an operator in
$\BB(\M)$, which is denoted by $S_i$.

The paper is organised as follows.\ In Section \ref{sec:programme}, we present some interesting results concerning spectral gaps.\ $\!$We show that if $A$ possesses spectral gaps and the sequence $\{A^{-n}TA^n\}_{n\in\ZZ}$ is norm-bounded, then $T$ is block diagonal with respect to a suitable orthogonal decomposition. We also prvide an alternative proof of Radjabalipour's theorem.\ In Section \ref{sec:extension}, we extend Radjabalipour's theorem to certain rich classes of operators, including invertible normal operators and invertible weighted sums of projections.\ Moreover, we prove a variant of Radjabalipour's theorem for matrices.

\section{Spectral gap and block representation}\label{sec:programme}

We begin this section with following key definition.

\begin{definition}\label{def:4.1}
Let ${A\in\mathbb{B}(\H)}$ be a self-adjoint operator.\ We say that $A$ has a
{\it spectral gap} if there exist disjoint nonempty closed sets
${\Delta_1,\Delta_2\subseteq\sigma(A)\subset\RR}$ such that
${\sigma(A)=\Delta_1\cup\Delta_2}$ and ${\sup\Delta_1<\inf\Delta_2}$.
\end{definition}

Although the following fact is likely known in the literature, we have been
unable to find any reference for it. We therefore prove it for the sake of
convenience.

\begin{lemma}\label{lem:4.2}
Let ${A\in\BB(\H)}$ be self-adjoint.\ Then the following are equivalent.
\begin{itemize}
\item[\rm(i)\,]
$A$ has a spectral gap;
\item[\rm(ii)]
there exist nonzero orthogonal reducing closed subspaces ${\H=\H_1\oplus\H_2}$
and real numbers ${\alpha\kern-1pt<\kern-1pt\beta}$ such that
${A_1\!\le\alpha I_{\H_1}}\!$ and ${A_2\!\ge\beta I_{H_2}},\!$ where
$A_j\kern-1pt$ is the~restric\-tion of $A$ to $\H_j\kern-1pt$ for ${j=1,2}$.
\end{itemize}
\end{lemma}

\begin{proof}
(i)$\Rightarrow$(ii).
Suppose ${\sigma(A)=\Delta_1\cup\Delta_2}$, where ${\Delta_1,\Delta_2}$ are
disjoint nonempty closed sets such that ${\sup\Delta_1\!<\inf\Delta_2}.$
Take real numbers ${\alpha,\beta}$ with ${\alpha\!<\!\beta}$ for which
$\Delta_1\!\subset\!(-\infty,\alpha]$ and
${\Delta_2\!\subset\![\beta,\infty)}.$ By the Riesz decomposition theorem and the spectral theorem, ${\H=\H_1\oplus\H_2}$,
where $\H_j$ are reducing closed subspaces corresponding to the spectral sets
$\Delta_j$ and ${\sigma(A_j)=\Delta_j}.$ We have
\[
\|A_1x\|^2=\langle A_1^2 x,x\rangle\le\alpha^2\|x\|^2
\]
since $A_1^2$ is positive and bounded. Hence,
\[
\langle A_1 x,x\rangle\le\|A_1x\|\,\|x\|\le\alpha\langle x,x\rangle,
\]
whence $A_1\le\alpha I_{\H_1}.$ Of course, this inequality can be directly established by using functional calculus since $\sigma(A_1) \subset (-\infty,\alpha]$. A similar argument leads to
${\beta\,I_{\H_2}\leq A_2}.$

(ii)$\Rightarrow$(i).
Since $\H_1$ and $\H_2$ reduce $A$, we have ${A=A_1\oplus A_2}$, and therefore
$\sigma(A)={\sigma(A_1)\cup\sigma(A_2)}.$ Moreover,
${\sigma(A_1)\subset(-\infty,\alpha]}$ and
${\sigma(A_2)\subset[\beta,\infty).}$ Hence the sets $\sigma(A_1)$ and
$\sigma(A_2)$ are disjoint closed subsets of $\sigma(A)$ separated by the
interval $(\alpha,\beta)$, so $A$ has a spectral gap.\
\end{proof}

\begin{theorem}\label{th:4.3}
Let $\H$ be a Hilbert space, let ${T\in\BB(\H)}$ be an arbitrary operator, and
let ${A\in\BB(\H)}$ be a positive invertible operator having a spectral gap.\
If the family $\{A^{-n}TA^n\}_{n\in\ZZ}$ is bounded in the operator norm, then
${T=T_1\oplus T_2}$ is block diagonal with respect to the orthogonal
decomposition ${\H=\H_1\oplus\H_2}$ given in Lemma~\ref{lem:4.2}.
\end{theorem}

\begin{proof}
Let us adopt the notation from the proof of Lemma~\ref{lem:4.2} and prove the
nontrivial implication.\ With respect to the orthogonal decomposition
${\H=\H_1\oplus\H_2}$, we represent the operators $T$ and $S$ in their block
forms as
\[
T=\begin{bmatrix} T_{11} & T_{12} \\ T_{21} & T_{22} \end{bmatrix}
\quad\hbox{and}\quad
A=\begin{bmatrix} A_1 & 0 \\ 0 & A_2 \end{bmatrix},
\]
where ${T_{ij}=\pi_i T\iota_j\in\BB(\H_j,\H_i)}.$ Here
${\pi_j\!:\H_1\oplus\H_2\to\H_j}$ is the natural projection and
${\iota_j\!:\H_j\to\H_1\oplus\H_2}$ is the natural embedding for ${j=1,2}$;
see \cite[Lemma 2.1.2]{MO}$.$ Let ${A_i=A|_{\H_i}}\!$ be the restriction of $A$
to $\H_i$, which is an invertible positive operator on $\H_i$ for ${i=1,2}$
due to ${\sigma(A)=\sigma(A_1)\cup\sigma(A_2)}$, where
${\sigma(A_1)\!\subseteq(0,\alpha]}$ and
${\sigma(A_2)\!\subseteq[\beta,\|A\|]}$~with~${\alpha\!<\!\beta}$.

Our goal is to show that ${T_{12}=0} $ and ${T_{21}=0}$.

According to the hypothesis, ${\|A^{-n}T A^n\|\leq M}$ for some ${M<\infty}$
for all ${n\in\ZZ.}$ It follows from
\[
A^{-n}TA^n
=\begin{bmatrix}
A_1^{-n}T_{11}A_1^n\; & \;A_1^{-n}T_{12}A_2^n \\
A_2^{-n}T_{21}A_1^n\; & \:A_2^{-n}T_{22}A_2^n
\end{bmatrix},
\]
and \cite[Lemma 2.1.2]{MO} that
$$
\|A_1^{-n}T_{12}A_2^n\|\le\|A^{-n}TA^n\|\leq M
\quad \mbox{for all}\quad
n\in\ZZ.                                                \eqno{(2.1)}
$$

To reach a contradiction, assume that ${T_{12}\ne0}.$ By the definition of
the operator norm, there exists a unit vector ${y\in\H_2}$ such that
${\|T_{12}y\|\!>\!\frac{\|T_{12}\|}{2}}.$ \hbox{For each ${n\in\ZZ}$, we}
normalise the vector ${A_2^{-n}y}$ by considering
${x_n\!=\!\frac{A_2^{-n}y}{\|A_2^{-n}y\|}\!\in\!\H_2}.$ Hence,
${A_2^nx_n\!=\!\frac{y}{\|A_2^{-n}y\|}}\!.$

On $\H_2$, we have ${A_2\ge\beta I_{H_2}>0}$, and so
${0\le A_2^{-n}\le\beta^{-n}I_{\H_2}}.$ This yields
\[
\|A_2^nx_n\|
=\frac{\|y\|}{\|A_2^{-n} y\|}
=\frac{1}{\|A_2^{-n}y\|}
\ge\frac{1}{\|A_2^{-n}\|}
\ge\beta^{\,n},
\]
whence
\[
\|T_{12}A_2^n x_n\|
=\frac{\|T_{12}y\|}{\|A_2^{-n}y\|}
\ge\frac{\|T_{12}\|}{2}\beta^{\,n}.
\]

On $\H_1$, we have ${A_1\!\le\alpha I_{\H_1}}$, and so
${A_1^{-2n}\!\ge\alpha^{-2n}I_{\H_1}}.$ For an arbitrary vector ${z\in\H_1}$,
we can write
\[
\langle A_1^{-n}z,A_1^{-n}z\rangle
=\langle A_1^{-2n}z,z\rangle\ge\langle\alpha^{-2n}z,z\rangle
=\langle\alpha^{-n}z,\alpha^{-n}z\rangle,
\]
and derive ${\|A_1^{-n}z\|\ge\alpha^{-n}\|z\|}.$ Taking ${z=T_{12}A_2^nx_n}$,
we infer that
\[
\|A_1^{-n}T_{12}A_2^n\|
\ge\|A_1^{-n}T_{12}A_2^nx_n\|
\ge\alpha^{-n}\|T_{12}A_2^n x_n\|
\ge\alpha^{-n}\frac{\|T_{12}\|}{2}\beta^n
=\frac{\|T_{12}\|}{2}\Big(\frac{\beta}{\alpha}\Big)^n\!.
\]
Since ${\beta/\alpha>1}$, we have ${\lim_{n\to -\infty}(\beta/\alpha)^n=0}$,
which contradicts (2.1).\ Therefore, ${T_{12}=0}.$ Repeating the
argument we arrive at ${T_{21}=0}.$ Thus,
\[
T=\begin{bmatrix} T_{11} & 0 \\ 0 & T_{22} \end{bmatrix}
=T_{11}\oplus T_{22}
\]
is block diagonal.
\end{proof}

In his result, Radjabalipour decomposes the Hilbert space into infinitely many spectral pieces and proves range inclusions for half-line projections, and then uses operator matrix techniques with infinite matrices and unilateral shift estimates. Now, we give an alternative proof of Theorem~\ref{thm:1.1}, in which we work with the full spectral measure and prove commutation with all spectral projections at once, and then utilize the monotone class theorem to extend from half-lines to all Borel sets. Our proof uses an idea similar to that of spectral gaps.

Recall that a monotone class is a nonempty family of subsets of a given set that is closed under countable increasing unions and countable decreasing intersections, and a $\pi$-system is a nonempty family of sets that is closed under finite intersections. The monotone class theorem states that if $\mathcal{P}$ is a $\pi$-system and $\mathcal{C}$ is a monotone class containing $\mathcal{P}$, then $\mathcal{C}$ contains the $\sigma$-algebra generated by $\mathcal{P}$; see \cite[Section 1.3]{LIE}.

\begin{theorem}\label{th:bounded-orbits}
Let $\mathcal H$ be a Hilbert space and let $A\in\mathbb B(\mathcal H)$ be positive and invertible.
For $T\in\mathbb B(\mathcal H)$, the following conditions are equivalent:
\begin{enumerate}
\item $\displaystyle \sup_{n\in\mathbb Z}\|A^{-n}TA^n\|<\infty$.
\item $AT=TA$.
\end{enumerate}
\end{theorem}

\begin{proof}
We may assume that $A$ is not a scalar multiple of the identity. Assume that $M:=\sup_{n\in\mathbb Z}\|A^{-n}TA^n\|<\infty$. Let $E:\mathcal B(\sigma(A))\longrightarrow \mathbb B(\mathcal H)$ be the spectral measure of the positive invertible operator $A$. Thus, $A=\int_{\sigma(A)}\lambda\,dE(\lambda)$. Since $A$ is positive and invertible, there exist constants $0<a<b<\infty$ such that $\sigma(A)\subseteq [a,b]$.

For every Borel set $\Omega\subseteq\sigma(A)$, put $P_\Omega=E(\Omega)$. Each $P_\Omega$ commutes with $A$ and hence with every integer power $A^n$.

We first prove that spectral projections corresponding to separated Borel sets are orthogonal through $T$. Let $\Omega_1,\Omega_2\subseteq\sigma(A)$ be Borel sets such that $\sup\Omega_1<\inf\Omega_2$. Set $\beta=\sup\Omega_1$ and $\alpha=\inf\Omega_2$. Then $0<a\leq \beta<\alpha\leq b$.

For every $n\geq0$, using the commutation of spectral projections with powers of $A$, we have
\[
P_{\Omega_2}TP_{\Omega_1}
= P_{\Omega_2}A^{-n}(A^nTA^{-n})A^nP_{\Omega_1}.
\]
Therefore,
\[
\|P_{\Omega_2}TP_{\Omega_1}\|
\leq \|P_{\Omega_2}A^{-n}\|\,\|A^nTA^{-n}\|\,\|A^nP_{\Omega_1}\| .
\]
Because the spectrum of $A$ restricted to $P_{\Omega_2}\mathcal H$ is contained in $[\alpha,b]$, and the spectrum restricted to $P_{\Omega_1}\mathcal H$ is contained in $[a,\beta]$, we get $\|P_{\Omega_2}A^{-n}\|\leq \alpha^{-n}$ and $\|A^nP_{\Omega_1}\|\leq \beta^n$. Consequently,
\[
\|P_{\Omega_2}TP_{\Omega_1}\|
\leq M\left(\frac{\beta}{\alpha}\right)^n .
\]
Since $\beta/\alpha<1$, letting $n\to\infty$ gives $P_{\Omega_2}TP_{\Omega_1}=0$.

Applying a similar argument to $A^{-1}$ (equivalently, replacing $n$ by $-n$ in the above estimate) yields $P_{\Omega_1}TP_{\Omega_2}=0$. Hence, whenever two Borel subsets of $\sigma(A)$ are separated by a positive distance, the corresponding off-diagonal terms of $T$ vanish.

Now fix $t\in\mathbb R$ and define $P_t=E((-\infty,t]\cap\sigma(A))$. We claim that $TP_t=P_tT$. For $\varepsilon>0$, put $\Omega_1=(-\infty,t]\cap\sigma(A)$ and $\Omega_2=[t+\varepsilon,\infty)\cap\sigma(A)$. These sets are separated, and therefore $E([t+\varepsilon,\infty)\cap\sigma(A))TP_t=0$ and $P_tTE([t+\varepsilon,\infty)\cap\sigma(A))=0$. As $\varepsilon\downarrow0$,
\[
E([t+\varepsilon,\infty)\cap\sigma(A))
\xrightarrow{\mathrm{SOT}}
E((t,\infty)\cap\sigma(A)).
\]
Hence, $E((t,\infty)\cap\sigma(A))TP_t=0$ and $P_tTE((t,\infty)\cap\sigma(A))=0$. Since $P_t+E((t,\infty)\cap\sigma(A))=I$, we obtain
\[
TP_t
=(P_t+E((t,\infty)\cap\sigma(A)))TP_t
=P_tTP_t,
\]
and
\[
P_tT
=P_tT(P_t+E((t,\infty)\cap\sigma(A)))
=P_tTP_t.
\]
Thus, $TP_t=P_tT$ for every $t\in\mathbb R$.

Now define $\mathcal C=\{\Omega\in\mathcal B(\sigma(A)):TE(\Omega)=E(\Omega)T\}$. We show that $\mathcal C$ is a monotone class. Suppose that $\Omega_n\uparrow\Omega$ with $\Omega_n\in\mathcal C$. The spectral theorem gives $E(\Omega_n)\xrightarrow{\mathrm{SOT}}E(\Omega)$. For every $\xi\in\mathcal H$, $TE(\Omega_n)\xi\longrightarrow TE(\Omega)\xi$. Moreover, $TE(\Omega_n)\xi=E(\Omega_n)T\xi\longrightarrow E(\Omega)T\xi$. Hence, $TE(\Omega)\xi=E(\Omega)T\xi$, and therefore $TE(\Omega)=E(\Omega)T$. The decreasing case follows in the same way. Thus $\mathcal C$ is a monotone class.

The family $\mathcal P=\{(-\infty,t]\cap\sigma(A):t\in\mathbb R\}$ is a $\pi$-system generating $\mathcal B(\sigma(A))$ since $\mathcal{B}(\mathbb{R}$ is $\sigma$-algebra generated by $\mathcal P=\{(-\infty,t]:t\in\mathbb R\}$. We have already proved that $\mathcal P\subseteq\mathcal C$. Therefore, by the monotone class theorem, $\mathcal C=\mathcal B(\sigma(A))$. Consequently, $TE(\Omega)=E(\Omega)T$ for every Borel set $\Omega\subseteq\sigma(A)$.

Hence $T$ commutes with the spectral measure of $A$. Using the spectral representation $A=\int_{\sigma(A)}\lambda\,dE(\lambda)$, we conclude that $TA=AT$.

Conversely, if $AT=TA$, then $A^{-n}TA^n=T$ for all $n\in\mathbb Z$, and so $\sup_{n\in\mathbb Z}\|A^{-n}TA^n\|=\|T\|<\infty$. 
\end{proof}
\section{Extensions to Several Rich Classes of Operators}\label{sec:extension}

In this section, we extend Theorem~\ref{thm:1.1} from
positive operators to normal \hbox{operators}.\

\begin{proposition}\label{t2}
Let ${A\in\mathbb{B}(\H)}$ be an invertible normal operator and let
$T\in\BB(\H)$ be an arbitrary operator. Then
\begin{itemize}
\item[\rm(a)]
${T|A|=|A|T}$ if and only if ${{\sup}_{n\in\ZZ}\|A^{-n}TA^n\|<\infty}$;
\item[\rm(b)]
if ${\sup_{n\in\mathbb{Z}}\|A^{-n}TA^n\|<\infty}$, then $TA$ is unitarily
equivalent to $A\,T$.
\end{itemize}
\end{proposition}

\begin{proof}
First, recall that $A$ is quasinormal if and only if ${(A^*A-AA^*)A=0}.$ Then
every normal operator is quasinormal, and every invertible quasinormal operator
is normal, so that invertible normal means invertible quasinormal.\ Moreover,
the partial isometry of the polar decomposition of an invertible operator is a
unitary operator (see, e.g., \cite[Corollary 2, p. 75]{HAL}).\ Furthermore, the
partial isometry of the polar decomposition of an operator $A$ commutes with
$|A|$ if and only if $A$ is quasinormal (see, e.g.,
\cite[Problem 137, p. 75, and Solution 137, p. 262]{HAL}).

Let $A$ be an invertible normal operator.\ Take its polar decomposition
${A=U|A|}=|A|U$ with ${U\!\in\BB(\H)}$ unitary and $|A|$ invertible with
${|A|^{-1}=A^{-1}U=U A^{-1}}$.\ Hence,
\[
A^n=U^n|A|^n=|A|^nU^n
\quad\hbox{and so}\quad
|A|^n=U^{-n}A^n=A^nU^{-n}
\]
for every ${n\in\mathbb{Z}}$. Thus, $\{A^{-n}TA^n\}$ and $\{|A|^{-n}T|A|^n\}$
are bounded together because
$$
\|A^{-n}TA^n\|=\|U^{*n}|A|^{-n}T|A|^nU^n\|=\|\,|A|^{-n}T|A|^n\|
\quad\hbox{for every}\quad
n\in\ZZ.                                                  \eqno{(*)}
$$

(a) If ${|A|T=T|A|}$, then ${|A|^{-1}T=T|A|^{-1}}$, so that
${|A|^{n}T=T|A|^{n}}$ and hence $T={|A|^{-n}T|A|^{n}}$ for every
${n\in\mathbb{Z}}.$ Therefore, $\{A^{-n}TA^n\}$ is bounded by $(*).$

Conversely, if $\{A^{-n}TA^n\}$ is bounded, then so is $\{|A|^{-n}T|A|^n\}$ by
$(*).$ Since~the positive operator $\kern-1pt|A|$ is invertible, it follows by
$\kern-1pt$Theorem~\ref{thm:1.1} (i.e., by \cite[\hbox{Corollary}~3]{RAD})
that ${|A|T=T|A|}.$

(b) Recall:\ ${|A|=U^{-1}A=A U^{-1}}$. If ${|A|T=T|A|}$, then ${UTA=A TU}$
so that $TA$ is unitarily equivalent to $A\,T$, and the claimed result comes
from item (a).
\end{proof}

\begin{remark}
It follows from Proposition \ref{t2} that an invertible operator $T$ is
quasinormal if and only if the sequence $\{(T^*T)^{-n}T(T^*T)^n\}$ is bounded.\
\end{remark}

One of the most powerful tools for simplifying problems with two operators is
to employ $2\times 2$ block matrices of operators; see \cite{MO, TIA}.\

\begin{corollary}
Suppose that ${A,B\kern1pt\in\kern1pt\BB(\H)}$ are invertible normal operators
and ${T\in\BB(\H)}$ is an arbitrary operator.\ Then ${|A|T=T|B|}$ if and only
if both sequences $\{A^{-n}TB^n\}_{n\in\ZZ}$ and $\{B^{-n}T^*A^n\}_{n\in\ZZ}$
are bounded.\
\end{corollary}

\begin{proof}
Set
${\,C=\begin{bmatrix}A & 0 \\ 0 & B\end{bmatrix}}$
and
${\,S=\begin{bmatrix}0 & T \\ T^* & 0\end{bmatrix}}.$
Then
${\,|C|=\begin{bmatrix}|A| & 0 \\ 0 & |B|\end{bmatrix}}.$
Therefore, ${|A|T=T|B|}$ if and only if ${S|C|=|C|S}.$
\vskip2pt
On the other hand, for any $n\in\NN$, we have
\begin{align*}
\left\|C^{-n}SC^n\right\|
&=
\left\|\begin{bmatrix}0 & A^{-n}TB^n \\ B^{-n}T^*A^n & 0\end{bmatrix}\right\|\\
&=\left\|\begin{bmatrix}0 & A^{-n}TB^n \\ B^{-n}T^*A^n & 0\end{bmatrix}
\begin{bmatrix}0 & I \\ I & 0\end{bmatrix} \right\|\\
&\qquad\qquad(\text{since the operator norm is a unitarily invariant norm})\\
&=\left\|\begin{bmatrix}A^{-n}TB^n &0\\ 0& B^{-n}T^*A^n\end{bmatrix}\right\|\\
&=\max\left\{\left\|A^{-n}TB^n\right\|, \left\|B^{-n}T^*A^n\right\|\right\}.
\end{align*}
The proof is completed by using Theorem \ref{thm:1.1} with $C$ and $S$ instead
of $A$ and $T.$
\end{proof}

The following theorem extends Radjabalipour's result to the rich class of
normal operators having pure point spectrum.\ It includes the case of normal
operators acting on finite-dimensional Hilbert spaces, that is, normal
matrices.\

\begin{theorem}
Let $\H$ be a Hilbert space, and let ${Ax=\sum_{j\in J}\lambda_jP_jx}$ for
${x\in\H}$ be an invertible weighted sum of projections, where
${\{P_j\}_{j\in J}}$ is a resolution of the identity on $\H$ and
$\{\lambda_j\}_{j\in J}$ is a bounded $($above and below$)$ family of complex
numbers.\ Let ${T\in\BB(\H)}$ be a bounded linear operator.\ If the sequence
$\{A^{-n}T A^n\}_{n\in\ZZ}$ is bounded in the operator norm, then the operator
$|A|$ commutes with $T.$
\end{theorem}

\begin{proof}
Let ${A=U|A|}$ be the polar decomposition of $A.$ Since $A$ is invertible, the
partial isometry $U$ is a unitary operator.\ Since $A$ is normal, it is
quasinormal.\ Thus, ${U|A|=|A|U}$ \cite[Problem 137]{HAL} (see also
\cite[Theorem 2.3.2.6]{FUR}).\ Therefore,
\[
A^{-n} T A^n = U^{-n} |A|^{-n} T |A|^n U^n.
\]

Note that $\{P_j\}_{j\in J}$ is a family of mutually orthogonal projections
such that ${P_j P_k=\delta_{jk} P_j}$ and ${\sum_{j\in J}P_j=I}$ in the strong
operator topology.\ Hence, by uniqueness of the square root, and considering
the weighted sum of projections for the inverse of $A$,
\[
|A| = \sum_{j \in J} |\lambda_j| P_j \quad \text{and} \quad |A|^n 
= \sum_{j \in J} |\lambda_j|^n P_j \quad (n \in \ZZ)
\]
in the strong operator topology.\

If some of the $|\lambda_j|$'s are equal, then we can combine the corresponding
$P_j$'s.\ In other words, we can group some blocks.\ So let ${J'\subseteq J}$
be the set of all indices $j$ such that the $|\lambda_j|$'s are distinct, and
for each ${j\in J'}\!$, let $P'_j$ be the sum of all orthogonal projections
$P_i$'s such that ${|\lambda_i|=|\lambda_j|}.$ Then the $P'_j$'s are still
orthogonal projections with sum equal to the identity operator $I.$

Now suppose that ${k,l\in J'}$ are distinct.\ Then
${|\lambda_k|\neq|\lambda_l|}.$ Without loss of generality, assume
${|\lambda_k|>|\lambda_l|}.$
\begin{itemize}
\item [(i)]
For the block $P'_l T P'_k$:
\[
\|\lambda_l^{-n} \lambda_k^n P'_l T P'_k\|
= \|P'_l (|A|^{-n} T |A|^n) P'_k\| \leq \||A|^{-n} T |A|^n\|,
\]
whence we conclude that the sequence $\{\lambda_l^{-n}\lambda_k^n P'_lTP'_k\}$
is bounded in the norm topology.\ Since ${|\lambda_k|>|\lambda_l|}$, we have
${|\lambda_l^{-n}\lambda_k^n|=\bigl(|\lambda_k|/|\lambda_l|\bigr)^n\to\infty}$
as ${n\to+\infty}.$ If ${P'_lTP'_k\neq 0}$, then
\[
\|\lambda_l^{-n} \lambda_k^n P'_l T P'_k\| \to \infty
\]
as ${n\to+\infty}$, contradicting boundedness. Hence, ${P'_lTP'_k=0}.$
\item [(ii)]
For the block ${P'_kTP'_l}$:
\[
\|\lambda_k^{-n} \lambda_l^n P'_k T P'_l\|
= \|P'_k (|A|^{-n} T |A|^n) P'_l\| \leq \||A|^{-n} T |A|^n\|.
\]
Substituting ${m=-n}$, as ${n\to-\infty}$ we have ${m\to+\infty}$ and
\[
|\lambda_k^{-n} \lambda_l^n|
= \bigl(|\lambda_k|/|\lambda_l|\bigr)^m \to \infty.
\]
Thus, if ${P'_kTP'_l\neq0}$, then
${\|\lambda_k^{-n}\lambda_l^nP'_kTP'_l\|\to\infty}$ as ${n\to-\infty}$, again
a contradiction.\ Hence, ${P'_kTP'_l=0}.$
\end{itemize}
Therefore, ${P'_kTP'_l=0}$ for all distinct ${k,l\in J'}\!.$ It follows that
\[
T = \sum_{j \in J'} P'_j T P'_j
\]
in the strong operator topology.\ In other words, with respect to the
orthogonal decomposition
\[
\H = \bigoplus_{j \in J'} P'_j(\H),
\]
the operator ${T={\rm diag}\bigl(P'_jTP'_j\bigr)_{j\in J'}}$ is block
diagonal.\ On the other hand, $|A|={\sum_{j \in J'}\!|\lambda_j|P'_j}\!$ is
diagonal in the same decomposition.\ Thus, $\!|A|$ and $T\!$
commute.\
\end{proof}

If all $|\lambda_j|$'s are distinct, we have ${J=J'}$, and so we
get the following result.\

\begin{theorem}\label{matrixmox}
Let $\H$ be a Hilbert space, and let ${Ax=\sum_{j\in J}\lambda_jP_jx}$ for
${x\in\H}$ be an invertible weighted sum of projections, where
$\{P_j\}_{j\in J}$ is a resolution of the identity on $\H$ and
$\{\lambda_j\}_{j\in J}$ is a bounded family of complex numbers with distinct
moduli.\ Let ${T\in\BB(\H)}$ be a bounded linear operator.\ If the sequence
${\{A^{-n}TA^n\}_{n\in\ZZ}}$ is bounded in the operator norm, then with respect
to the orthogonal decomposition ${\H=\bigoplus_{j\in J}P_j(\H)}$, the operator
$T$ is block diagonal and $A$ commutes with $T.$
\end{theorem}

\begin{remark}
A typical example of matrices described in Theorem \ref{matrixmox} is the
matrix ${A=UD\,U^*}$, where $D$ is in full matrix algebra $\MM_N$ of all
complex ${N\times N}$ matrices is a diagonal matrix of the form
${\lambda_1I_1\oplus\cdots\oplus\lambda_kI_k}$, where the $\lambda_i$'s have
distinct moduli and ${I_j\in\MM_{n_j}}$ is the identity matrix for each
${1\leq j\leq k}$ with ${n_1+\cdots+n_k=N}$, and ${U\in\MM_N}$ is a unitary
matrix of the form ${U_1\oplus\cdots\oplus U_k}$ in which
${U_j\in\MM}_{n_j}$ is unitary for each ${1\leq j\leq k}.$ Furthermore, the
eigenvalues of $A$ have distinct moduli if and only if the characteristic
polynomial $p(t)$ of $|A|$ and its derivative $p'(t)$ have no common factor.\
\end{remark}

The next theorem provides a characterization of the cone of positive definite
matrices in the center of $\MM_N.$

\begin{theorem} \label{t3}
A positive definite matrix $A\in \mathbb{M}_N$ is a positive multiple of the
identity matrix if and only if the sequence $\{A^{-n}TA^n\}_{n\in\ZZ}$
is bounded for all ${T\in\MM_N}$.
\end{theorem}

\begin{proof}
The proof of ($\Longrightarrow$) is clear.\ We now present several proofs for
($\Longleftarrow$).\

\textbf{First proof.}

Set ${a=\min\{\langle Az,z\rangle\!: \|z\|=1\}}$ and
${b=\max\{\langle Az, z\rangle\!: \|z\|=1\}}.$ Therefore, the numerical range
${W(A)=\{\langle Az,z\rangle\!: z\in \CC^N\!,\|z\|=1\}}$ of $A$, which is
convex, is ${[a,b]}\kern-1pt\subset\kern-1pt{(0,\infty)}.$ It is known that
${a,b\in W(A)}$; see \cite[Theorem 2.1.1]{BDMP}.\ Choose unit eigenvectors $x$
and $y$ such that ${Ax=ax}$ and ${Ay=by}.$ Take a matrix $T$ such that
${\langle Ty,x\rangle\neq 0}.$ Since the sequence
\[
\{\langle (A^{-n} T A^n)y, x\rangle\}_{n\in\mathbb{Z}}
=\{\langle T (b^ny), a^{-n}x\rangle\}_{n\in\mathbb{Z}}
\]
is bounded, it follows that the sequence $\left\{(b/a)^n\right\}_{n\in\ZZ}$ is
bounded.\ Hence, ${b\leq a}$, and so ${a=b}.$ Thus, ${A=aI}.$

\textbf{Second proof.}

All norms on the finite-dimensional Hilbert space $\mathbb{C}^N$ are
equivalent.\ Hence, we can assume that $\{A^{-n}TA^n\}_{n\in\ZZ}$ is bounded
in the Frobenius norm $\|\cdot\|_F$, also known as the Hilbert--Schmidt norm.\
Since $A$ is positive definite, there exists an orthonormal basis
${\{z_1,\ldots,z_N\}}$ for $\CC^N$ such that ${Az_j=\lambda_jz_j}$ for
${j=1,\ldots,N}$, where $\lambda_j$'s are the (positive) eigenvalues of $A.$
Evidently, $A^{n}z_j=\lambda_j^n z_j$ and $A^{-n}z_j=\lambda_j^{-n}z_j.$ We
have
\[
\|A^{-n}TA^n\!\|_F^2
=\!\!\sum_{i,j=1}^N\!|\langle A^{-n}TA^nz_i,z_j\rangle|^2\!
=\!\!\sum_{i,j=1}^N\!\lambda_i^{2n}\,\lambda_j^{-2n}
|\langle Tz_i,z_j\rangle|^2\!
\geq\!\big(\smallfrac{\lambda_i}{\lambda_j}\big)^{2n}
|\langle Tz_i,z_j\rangle|^2
\]
for ${1\leq i,j\leq N}.$ Take a matrix $T$ such that all of its entries
represented in the basis ${\{z_1,\ldots,z_N\}}$ are nonzero.\ For a specific
pair $(i,j)$, it follows from the boundedness of the conjugate orbit that, as
${n\to\infty}$, we get ${\lambda_i\leq \lambda_j}$, and as ${n\to-\infty}$, we
get ${\lambda_j\leq\lambda_i}.$ Hence, ${\lambda_i=\lambda_j=\lambda}$ for
some positive real number $\lambda$ and all ${i,j=1,\ldots,N}.$ Therefore,
${A=\lambda I}.$

\textbf{Third proof.}

Without loss of generality, we can assume that
${A=\operatorname{diag}(\lambda_1,\ldots,\lambda_N)}$, where $\lambda_i$'s are
the (positive) eigenvalues of $A.$ If $T\!=[t_{ij}]$, then
${A^{-n}TA^n\!=[(\lambda_j/\lambda_i)^n t_{ij}]}$~for~$n={1,2,\ldots}.$ Taking
a matrix $T$ with nonzero entries, the boundedness of the conjugate orbit
implies that the sequences $\{(\lambda_j/\lambda_i)^n\}_{n\in\ZZ}$ are bounded
for ${i,j=1,2,\ldots,N}.$ This entails that ${\lambda_i=\lambda_j}$ for all
${i,j=1,2,\ldots,N}.$ Hence, ${A=\lambda I}$ for some positive number
$\lambda$.
\end{proof}

\medskip
\noindent \textit{Author Contributions Statement.} All authors wrote, edited, and reviewed the manuscript.

\medskip
\noindent \textit{Conflict of Interest Statement.} On behalf of the authors, the corresponding author states that there is no conflict of interest.

\medskip
\noindent\textit{Data Availability Statement.} Data sharing is not applicable to this article as no datasets were generated or analyzed during the current study.

\medskip
\noindent \textit{Funding Declaration.} This research received no funding.

\medskip
\bibliographystyle{amsplain}

\end{document}